\documentclass[a4paper,12pt]{article}
\usepackage{latexsym}
\usepackage{amsfonts}
\usepackage{amssymb}
\usepackage{amsmath}
\usepackage{graphicx}
\usepackage{makeidx}
\usepackage{textcomp}
\usepackage{multirow}
\usepackage{amsthm}

\newtheorem{theorem}{Theorem}[section]
\newtheorem{lemma}[theorem]{Lemma}
\newtheorem{example}[theorem]{Example}
\newtheorem{proposition}[theorem]{Proposition}
\newtheorem{corollary}[theorem]{Corollary}

\theoremstyle{definition} \newtheorem{definition}[theorem]{Definition}
\theoremstyle{definition} \newtheorem{hyp}[theorem]{Hypothesis}
\theoremstyle{definition} \newtheorem{remark}{Remark}

\title{Python Implementation and Construction of \\ Finite Abelian Groups}

\author{Paul Bradley\\
\small Mathematics Department\\[-0.8ex]
\small Queen Mary's Grammar School\\[-0.8ex] 
\small Walsall, WS1 2PG, U.K.\\
\small\tt pm-bradley@qmgs.walsall.sch.uk\\
\and
John Smethurst\\
\small Mathematics Department\\[-0.8ex]
\small Queen Mary's Grammar School\\[-0.8ex] 
\small Walsall, WS1 2PG, U.K.\\
\small\tt j-smethurst@qmgs.walsall.sch.uk\\
}
\begin{document}
\maketitle
\begin{abstract}
Here we present a working framework to establish finite abelian groups in python. The primary aim is to allow new A-level students to work with examples of finite abelian groups using open source software. We include the code used in the implementation of the framework. We also prove some useful results regarding finite abelian groups which are used to establish the functions and help show how number theoretic results can blend with computational power when studying algebra. The groups we establish are based on modular multiplication and addition. We include direct products of cyclic groups meaning the user has access to all finite abelian groups.
\end{abstract}

\section{Introduction}

New A-level Mathematics syllabi including the topics of technology in mathematics and the addition of early group theory have given schools a problem and an opportunity (for example see OCR~\cite{MEI}). Here we use technology beyond the standard software to develop a framework for group theory. This will allow students to explore the programming language \textsc{Python} and use computers to process information. The end result being more concrete examples of groups for them to explore. Giving pupils the opportunity to view groups as more than abstract ideas and generate groups in a variety of ways to explore the concept of isomorphism will also be useful. 

In this paper we present the inputs we  have used and explore the group theoretic ideas behind them. The framework will allow pupils to generate abelian groups using addition modulo $n$ and multiplication modulo $n$. They will be able to determine group orders, element orders, inverses, generate subgroups and find isomorphisms between two groups. They will also be able to determine the group structure and torsion coefficients for their groups. We supply a mixture of group theoretic approaches and justifications for some of the ideas students will meet alongside the \textsc{Python} code and commands. Some of the number theoretic approaches presented here go beyond the scope of the A-level courses we are considering. They act as bridges between computational group theory and theoretical group theory.

The authors acknowledge the existence of other algebra computation systems, for example \textsc{Magma.} These are far less frequently used in secondary schools. In many cases they will require students and staff to learn a new language. For ease of reference and reading when we give an example script for implementation we will usually include the output in an Appendix.

\section{Building a group}

We make use of the relationships in $\mathbb{N}$ to construct our groups. 
 
\begin{definition}
Let $n \in \mathbb{N}$ and let $\oplus$ by addition modulo $n$ and $\otimes$ be multiplication modulo $n$.

We can then form two types of group. The first is $$(n, \oplus) = \{0, 1, ..., (n-1)\} \mbox{ equipped with addition modulo } n.$$ The second is 
$$(n, \otimes) = \{a \in \mathbb{N}\  \vert \ a<n,  gcd(a,n)=1\} \mbox{ with multiplication mod }n.$$

An important function we will make use of is Euler's Phi function, $\phi(n)$. For $n \in \mathbb{N}$ let $\phi(n)$ equal the number of natural numbers less than $n$ that are coprime to $n$.

\end{definition}
The implementation for these two groups is given below
\begin{example}\label{Ex1}
\begin{verbatim}
# Import the class Group from the file groups.py
from groups import Group

# Create an instance of the group (10,+)
G = Group("add",10)

# print the group elements
print "Group Elements = ",G.g()

Group Elements =  [0, 1, 2, 3, 4, 5, 6, 7, 8, 9]

# Create an instance of the group (15,x)
H = Group("mult",15)

# print the group elements
print "Group Elements = ",H.g()

Group Elements =  [1, 2, 4, 7, 8, 11, 13, 14]
\end{verbatim}
\end{example}
Next we calculate the order of each group for any given input $n$. For $(n, \oplus)$ this is trivial as the order is $(n-1).$ For $(n, \otimes)$ we have by definition that the order is equal to $\phi(n)$. This function is added to the file \textsc{Groups.py}.

\begin{lemma}
\label{EPF}
Let $m, n \in \mathbb{N}, m\neq 1,$ be such that $m \vert n$ and let $H \cong \mathbb{Z}_n.$ If $p_1, p_2,..., p_r$ are the distinct prime divisors of $m$, then the number of elements in $H$ of order $m$ is 
$$\phi(m) = \frac{m}{p_1 p_2...p_r} (p_1-1)(p_2-1)...(p_r-1).$$
\end{lemma}  
 
\begin{proof}
The elements of order $m$ in $H$ are all contained in a cyclic subgroup of order $m$. Therefore the number of elements in $H$ of order $m$ is equal to the number of $j \in \{ 1,2, ...,m\}$ which are coprime to $m$. This is the definition of $\phi(m).$
\end{proof}
\begin{remark}
The definition of $\phi$ given in Lemma~\ref{EPF} is given in \cite{Rosen}.
\end{remark}

\begin{lemma}
For $\ell \in (n, \oplus)$, the inverse element $\ell^{-1} = n-\ell$. 
\end{lemma}
\begin{proof}
As $\oplus$ is addition the identity element is $0$. The result follows immediately.
\end{proof}
\begin{lemma}
For $\ell \in (n, \otimes)$, the inverse element $\ell^{-1} = \ell^{(\phi(n)-1)}$. 
\end{lemma}
\begin{proof}
The identity element in the group $(n, \otimes)$ is $1$. As $gcd(\ell, n)=1$ follows from Euler's theorem that $\ell^{\phi(n)}\equiv \ell \mbox{ mod }n$. Hence the statement follows directly.
\end{proof}

The commands for group operations and inverses are given below.
\begin{example}\label{Ex2}
\begin{verbatim}
# Import the class Group from the file groups.py
from groups import Group

# Create an instance of the group (10,+)
print "Creating a group (10,+)"
G = Group("add",10)

# print the group components
Gg = G.g()
print "Group Elements = ",Gg
print "Group Identity = ",G.i
print "Group Order = ",G.o
print "Inverse of element 3 = ",G.inv(3)
print "Multiplication of 7 and 6 = ",G.op(7,6)
print "7 to the power 3 = ",G.pow(7,3)

Ggo = G.go()
GgoDict = dict(zip(Gg, Ggo))
for g in GgoDict:
    print "Element: ",g," has order: ",GgoDict[g]

\end{verbatim}
\end{example}
Secondly we perform the same functions on our group $H\cong (15, \otimes)$
\begin{example}\label{Ex3}
\begin{verbatim}
# Create an instance of the group (15,x)
print "Creating a group (15,x)"
H = Group("mult",15)

# print the group components
Hh = H.g()
print "Group Elements = ",Hh
print "Group Identity = ",H.i
print "Group Order = ",H.o
print "Inverse of element 2 = ",H.inv(2)
print "Multiplication of 2 and 10 = ",H.op(2,10)
print "2 to the power 10 = ",H.pow(2,10)
print "Group orders:"

Hho = H.go()
HhoDict = dict(zip(Hh, Hho))
for h in HhoDict:
    print "Element: ",h," has order: ",HhoDict[h]
\end{verbatim}
\end{example}
 
\section{Subgroups}

We wish to establish the orders of elements. We use the same process as generating the subgroup $\left< g \right>$ for $g \in G$. This is because the order of $g$, denoted $o(g)=|\left< g \right>|.$ There are many practical reasons that it would be preferable to generate this subgroup rather than simply calculate the order numerically. We are able to make use of the group $G$ being abelian when determining the $\left< S \right>$ for some subset $S\subseteq G$. 

\begin{definition}
Let $G$ be a group. Then for $g \in G$ let $o(g)$ denote the order of the group element $g$. Also let $1_G$ denote the identity element of $G$.
\end{definition}
\begin{lemma}\label{subgen}
For a non-empty subset of $S\subset G$, where $S=\{g_1, g_2,...,g_s\}$. Then if $a_i=o(g_i)$ then $\left< S \right>=\{g_1^{b_1} g_2^{b_2} ...g_s^{b_s}\ | \ 0\leq b_i\leq a_i \}$.
\end{lemma}
\begin{proof}
As $G$ is finite each element has finite order. Then without loss of generality we can state that $\left< S \right>$ is the set of all combinations of the elements in $S$. As there are infinitely many such combinations we can use commutativity to collect occurrences of $g_i$ in a string together. These will then cycle through the powers of $g_i$ until we reach the identity element $1_G$. The result follows.

\end{proof}

The implementation of generating a subgroup or checking whether or not a subset is a subgroup are very important skills. Hence this is left as an exercise for the reader. We do however provide the following guidance.

If we are in the situation of wishing to test whether or not a subset $S\subseteq G$ is a subgroup we have two options. The first is to generate the subgroup $\left< S \right> \leq G$ and to determine whether or not $S=\left< S \right> $. The second is to check the subgroup criterion which we now state without proof.

\begin{lemma}[Subgroup Criterion]
Let $G$ be a group. Let $S$ be a non-empty subset of $G$. Then $S \leq G$ if and only if for all $a, b \in S$ we have that $ab^{-1} \in S$.
\end{lemma}

The \textsc{Python} code produces the subgroup generated by a set of elements. That is $\left<S\right> \le G$ where $S\subseteq G$. This is obtained as shown in Lemma~\ref{subgen}.

We note that $\left<S\right>$ is the smallest subgroup of $G$ containing every element in $S$. We also say that if $\left< S \right> =G$ then $S$ is a generating set for $G$.

\begin{example}\label{Ex5}
 \begin{verbatim}
# Importing required modules
from groups import Group
from pure import setCrossProd

# Creating group and subgroup
print "Create (120,+)"
G = Group("add",120)

print "Create subgroup with elements 60,30 and 15"
g60cycle = G.gcycle(60)
print "60 creates",g60cycle
g30cycle = G.gcycle(30)
print "30 creates",g30cycle
g15cycle = G.gcycle(15)
print "15 creates",g15cycle
Gx = setCrossProd([g60cycle,g30cycle,g15cycle])
GSg = []
for x in Gx:
    s = G.op(x)
    GSg.append(s)
GSg = set(GSg)
GS = Group("add",120,GSg)

# print the subgroup components
GSg = GS.g()
print "Subgroup Elements = ",GSg
print "Subgroup Identity = ",GS.i
print "Subgroup Order = ",GS.o
print "Inverse of element 15 = ",GS.inv(15)
print "Multiplication of 15 and 30 = ",GS.op(15,30)
print "15 to the power 3 = ",GS.pow(15,3)

GSgo = GS.go()
for g,go in zip(GSg,GSgo):
    print "Element: ",g," has order: ",go
\end{verbatim}
\end{example}

\begin{example}\label{Ex6}
\begin{verbatim}
# Importing required modules
from groups import Group
from pure import setCrossProd

# Creating group and subgroup
print "Create (64,x)"
G = Group("mult",64)

print "Create subgroup with elements 17 and 7"
g17cycle = G.gcycle(17)
print "17 creates = ", g17cycle
g7cycle = G.gcycle(7)
print "7 creates = ",g7cycle
Gx = setCrossProd([g17cycle,g7cycle])
GSg = []
for x in Gx:
    s = G.op(x)
    GSg.append(s)
GSg = set(GSg)
GS = Group("mult",32,GSg)

# print the subgroup components
GSg = GS.g()
print "Subgroup Elements = ",GSg
print "Subgroup Identity = ",GS.i
print "Subgroup Order = ",GS.o
print "Inverse of element 25 = ",GS.inv(25)
print "Multiplication of 7 and 25 = ",GS.op(7,25)
print "18 to the power 3 = ",GS.pow(15,3)

GSgo = GS.go()
for g,go in zip(GSg,GSgo):
    print "Element: ",g," has order: ",go
\end{verbatim}
\end{example}

\section{Direct Products}
\begin{definition}
Let $G$ and $H$ be groups. Then it is possible to form a third group from these. One method for doing this is the Direct Product denoted $G\times H$. This is the group of elements which are ordered pairs $(g,h)$ where $g\in G, h\in H$. The group operation is point-wise multiplication. That is 
$$(g_1, h_1)*(g_2, h_2)=(g_1g_2, h_1h_2)$$
where $g_1g_2$ is calculated in $G$ and $h_1h_2$ is calculated in $H$.
\end{definition}

This allows us to create a great many more abelian groups from known groups. In fact we can create all finite abelian groups (see section 5 for further details). There is no restriction on the number of groups we include in the direct product. However determining whether or not we have a ``new" group is a non-trivial question. It should be clear that 
$$G\times H\cong H\times G$$
and that $$|G\times H|=|G| \times |H|.$$
\begin{example}\label{Ex8}
Here we form the direct product of two groups.
\begin{verbatim}
# Import the class Group from the file groups.py
from groups import Group

# Create an instance of the group (5,+)x(9,x)
print "Create a group (5,+)x(9,x)"
G = Group(["add","mult"],[5,9])

# print the group components
Gg = G.g()
print "Group elements:"
for ind in Gg:
    print "Element ",ind," is ",Gg[ind]
Gg = Gg.values()

print "Group Identity = ",G.i
print "Group Order = ",G.o

g3 = G.g(3)
print "Element 3 = ",G.g(3)
print "Inverse of element 3 = ",G.inv(g3)

g7 = G.g(7)
g6 = G.g(6)
print "Element 7 = ",g7
print "Element 6 = ",g6
print "Multiplication of element 7 and element 6=",G.op(g7,g6)
print "Element 7 to the power of 3 = ",G.pow(g7,3)

print "Group orders:"
Ggo = G.go()
for (g,go) in zip(Gg,Ggo):
    print "Element ",g," has order ",go
\end{verbatim}
\end{example}

\begin{example}\label{Ex7}
\begin{verbatim}
# Importing require modules
from groups import Group
from pure import setCrossProd

# Creating group and subgroup
print "Create (6,+)x(9,x)"
G = Group(["add","mult"],[6,9])
Gg = G.g()
print "Group elements:"
for ind in Gg:
    print "Element ",ind," is ",Gg[ind]

print "Create subgroup with elements 2 and 22"
g2 = G.g(2)
g22 = G.g(22)
g2cycle = G.gcycle(g2)
print "2 creates = ", g2cycle
g22cycle = G.gcycle(g22)
print "22 creates = ", g22cycle

Gx = setCrossProd([g2cycle,g22cycle])
GSg = []

for x in Gx:
    s = G.op(x)
    GSg.append(s)
GSg = [list(x) for x in set(tuple(x) for x in GSg)]
GS = Group(["add","mult"],[6,9],GSg)

# print the subgroup components
GSg = GS.g()
print "Subgroup elements:"
for ind in GSg:
    print "Element ",ind," is ",GSg[ind]
GSg = GSg.values()

print "Subgroup Identity = ",GS.i
print "Subgroup Order = ",GS.o

g3 = GS.g(3)
print "Element 3 = ",GS.g(3)
print "Inverse of element 3 = ",GS.inv(g3)

g7 = GS.g(7)
g6 = GS.g(6)
print "Element 7 = ",g7
print "Element 6 = ",g6
print "Multiplication of element 7 and element 6=",GS.op(g7,g6)
print "Element 7 to the power of 3 = ",GS.pow(g7,3)

print "Group orders:"
GSgo = GS.go()
for (g,go) in zip(GSg,GSgo):
    print "Element ",g," has order ",go
\end{verbatim}
\end{example}
\section{Isomorphism}
 
Our aim is to provide a method of determining whether two abelian groups generated are isomorphic. Here we make use of the Classification of Finite Abelian Groups.

It is known that any finite abelian group can be expressed as a direct product of cyclic subgroups of prime-power order. This will allow us to determine whether two groups are isomorphic.

\begin{proposition}
Let $G$ and $H$ be finite abelian groups such that both $G$ and $H$ have identical numbers of elements of any given order. Then $G \cong H$. 
\end{proposition}
\begin{proof}
As we can write $G$ as a product of cyclic groups of prime power order we can construct it from its list of element orders. If we take an element of largest prime power order on the list, say $p^a$, we can form the cyclic subgroup containing this element. The subgroup will contain $p^a-1$ non identity elements. If we now remove the corresponding number of elements which will be contained in this subgroup we have a reduced list. By Lemma~\ref{EPF} this is $\phi(p^a)$ elements of order $p^a$, $\phi(p^{a-1})$ elements of order $p^{a-1}$ and so on.
 We can repeat this process until we have all of the elements of prime power order have been removed. The direct product of the subgroups will be $G$. As we constructed this direct product only using the element orders it is true that we also constructed $H$.   
\end{proof}

This approach means when establishing isomorphism of two finite abelian groups of order $n$, it is enough to consider the number of $p$-elements for all $p|n$.

We wish to allow the system to compare the number of elements with prime power order to standard known groups. The Classification of Finite Abelian Groups tells us that for a finite abelian group $G$ we have
\begin{eqnarray}\label{DP}
G\cong \mathbb{Z}_{m_1} \times \mathbb{Z}_{m_2} \times....\times \mathbb{Z}_{m_k}
\end{eqnarray}
where $m_i |m_{i+1}$ for all $1\leq i\leq k-1.$

Moreover the list of $m_is$ is unique. This leaves us with the problem of enumerating the number of elements in $G\cong \mathbb{Z}_{m_1} \times \mathbb{Z}_{m_2} \times ....\times \mathbb{Z}_{m_k}$ with prime power order for some $p\vert |G|$.

\begin{definition}
We define a function $\lambda(a,b)$ to equal the minimum of $\{a,b\}$ for $a,b \in \mathbb{Z}$. Next we define two functions which will appear in our count. Let $a,m,p \in \mathbb{N}$ with $p$ a prime. Then we can write $m=p^b q$ where $p\nmid q,$ we note that $b=0$ is a possibility. We now define $\eta(p,a,m) = p^{\lambda(a,b)}$ and $\eta^-(p,a,m)= p^{\lambda(a-1,b)}$.

\end{definition}
The next theorem is provided to allow students to compute the number of elements of prime power order for a finite abelian group without having to construct the elements. In practice this will require more programming to achieve and is given here as a theoretic approach. Counting the number of elements of any given order in a group is a non-trivial task.
\begin{theorem}
\label{primecount}
Let $G$ be a finite abelian group of order $n$ and set $$G\cong \mathbb{Z}_{m_1} \times \mathbb{Z}_{m_2} \times ....\times \mathbb{Z}_{m_k}$$
where $m_i |m_{i+1}$ for all $1\leq i\leq k-1.$ Then let $p^a$ be a prime power dividing $n$. Then $G$ has precisely 
$$\prod_{i=1}^k \eta(p,a,m_i) - \prod_{i=1}^k \eta^-(p,a,m_i)$$

elements of order $p^a$.
\end{theorem}

\begin{proof}
Using \ref{EPF} we can count the number of elements with order $p^b$ for some $b\leq a$ from each cyclic group in our direct product of $G$. For an element of $g$ to have order $p^a$ it must take elements of order $p^b | p^a$ from each cyclic component and at least one element of order $p^a$. For each $m_i$ there are $\phi(p^b)$ elements of order $p^b$. This gives us a total count of elements of $\mathbb{Z}_{m_i}$ with relevant orders of $$(1+\phi(p)+\phi(p^2)+...+\phi(p^{\lambda(a,a_i)})).$$

\begin{eqnarray*}
 (1+\phi(p)+\phi(p^2)+...+\phi(p^{\lambda(a,a_i)})) & = & (1+ (p-1) \\
				& + & p(p-1)+...+p^{\lambda(a,a_i)-1}(p-1)) \\
  & = & \left( 1+(p-1)\frac{p^{\lambda(a,a_i)}-1}{p-1}\right) \\
   & = & p^{\lambda(a,a_i)}.
\end{eqnarray*}

There will therefore be $\prod_{i=1}^k \eta(p,a,m_i)$ elements with orders $p^b$ where $b\leq a$ in $G$.

However, we will need to remove from our count any combinations which will not contain an element of order $p^a$. Following the same strategy as above we see that there will be $p^{\lambda(a-1,a_i)}$ such contributions to the count for each cyclic subgroup $\mathbb{Z}_{m_i}$. This leads to the final statement.

\end{proof}

Using Theorem~\ref{primecount} we are able to establish lists of all finite abelian groups of order $n$ to test for isomorphism.

It is often preferable to write an abelian group given in the form of a direct product of cyclic groups in the form given in equation (\ref{DP}). We can achieve this using the following process.

Split each cyclic group in the product into cyclic groups of prime powers. Then combining each cyclic $p-$group of largest order for each prime, $p$, to give $\mathbb{Z}_{m_1}$. Repeat the process for the remaining cyclic groups to form $\mathbb{Z}_{m_2}$. Continue in this manner until all cyclic subgroups have been assigned.

The described method will allow us to enter orders of cyclic groups in a direct product and determine the torsion coefficients for that group.
\begin{example}
\begin{eqnarray*}
 \mathbb{Z}_{24} \times \mathbb{Z}_{32} \times \mathbb{Z}_{42} & \cong & (\mathbb{Z}_{8} \times \mathbb{Z}_{3}) \times (\mathbb{Z}_{32}) \times (\mathbb{Z}_{7} \times \mathbb{Z}_{3} \times \mathbb{Z}_2) \\
  & \cong & (\mathbb{Z}_{32} \times \mathbb{Z}_{7} \times \mathbb{Z}_{3}) \times (\mathbb{Z}_{8} \times \mathbb{Z}_{3}) \times (\mathbb{Z}_2) \\
  & \cong & 	\mathbb{Z}_{672} \times \mathbb{Z}_{24} \times \mathbb{Z}_{2} \\
 \end{eqnarray*}
\end{example}

If a student wishes to determine which finite abelian group they have generated from a product of cyclic groups they can use the above algorithm. However if they generate a group using multiplication mod $n$ they need to investigate further. We provide an example as an explanation of a process they can use.

\begin{example}\label{Ex4}
 \begin{verbatim}
EXAMPLE FOR DETERMINING ISOMORPHISM TYPE OF 
Multiplication mod 32
 
 1) Find all other candidates of order n.
 2) Generate the other candidates using products of cyclic
groups.
 3) List element orders for each candidate and the group.
    in sorted order.
 4) Compare lists.  

# Import the class Group from the file groups.py
from groups import Group

# Create an instance of the group (32,x)
G = Group("mult",32)
print "Testing group:"
print "(32,x) group order = ",G.o
Ggo = G.go()
print "(32,x) sorted element orders = ",sorted(Ggo),"\n"

# Creating candidate groups
G1 = Group("add",16)
G2 = Group(["add","add"],[8,2])
G3 = Group(["add","add","add"],[4,2,2])
G4 = Group(["add","add","add","add"],[2,2,2,2])

print "Candidate groups:"
G1go = G1.go()
print "(16,+) sorted element orders =
 ",sorted(G1go)
G2go = G2.go()
print "(8,+)x(2,+) sorted element orders =
 ",sorted(G2go)
G3go = G3.go()
print "(4,+)x(2,+)x(2,+) sorted element orders =
 ",sorted(G3go)
G4go = G4.go()
print "(2,+)x(2,+)x(2,+)x(2,+) sorted element orders
 = ",sorted(G4go),"\n"

print "Conclusion:"
print "(32,x) is therefore isomorphic to (8,+)x(2,+)"
\end{verbatim}
\end{example}
\begin{remark}
It is clear from the example that for large numbers of elements this can be difficult to compare groups. Hence we can restrict our checking to $p$-elements if required.
\end{remark}

\appendix
\section{Repository Links}

Here we provide the URL where the relevant \textsc{PYTHON} files can be located. The files required are \textsc{Pure.py} and \textsc{Groups.py}.
\begin{verbatim}
https://github.com/Smeths/pygroup
 \end{verbatim}
 
\section{Output}
\subsection{Example~\ref{Ex2}}
\begin{verbatim}
Creating a group (10,+)
Group Elements =  [0, 1, 2, 3, 4, 5, 6, 7, 8, 9]
Group Identity =  0
Group Order =  10
Inverse of element 3 =  7
Multiplication of 7 and 6 =  3
7 to the power 3 =  1
Element:  0  has order:  1
Element:  1  has order:  10
Element:  2  has order:  5
Element:  3  has order:  10
Element:  4  has order:  5
Element:  5  has order:  2
Element:  6  has order:  5
Element:  7  has order:  10
Element:  8  has order:  5
Element:  9  has order:  10
\end{verbatim}
 
\subsection{Example~\ref{Ex3}}
\begin{verbatim}
Creating a group (15,x)
Group Elements =  [1, 2, 4, 7, 8, 11, 13, 14]
Group Identity =  1
Group Order =  8
Inverse of element 2 =  8
Multiplication of 2 and 10 =  5
2 to the power 10 =  4
Group orders:
Element:  1  has order:  1
Element:  2  has order:  4
Element:  4  has order:  2
Element:  7  has order:  4
Element:  8  has order:  4
Element:  11  has order:  2
Element:  13  has order:  4
Element:  14  has order:  2
\end{verbatim}
\subsection{Example~\ref{Ex5}}
\begin{verbatim}
Create (120,+)
Create subgroup with elements 60,30 and 15
60 creates [60, 0]
30 creates [30, 60, 90, 0]
15 creates [15, 30, 45, 60, 75, 90, 105, 0]
Subgroup Elements =  set([0, 105, 75, 45, 15, 90, 60, 30])
Subgroup Identity =  0
Subgroup Order =  8
Inverse of element 15 =  105
Multiplication of 15 and 30 =  45
15 to the power 3 =  45
Element:  0  has order:  1
Element:  105  has order:  8
Element:  75  has order:  8
Element:  45  has order:  8
Element:  15  has order:  8
Element:  90  has order:  4
Element:  60  has order:  2
Element:  30  has order:  4
\end{verbatim}
\subsection{Example~\ref{Ex6}}
\begin{verbatim}
Create (64,x)
Create subgroup with elements 17 and 7
17 creates =  [17, 33, 49, 1]
7 creates =  [7, 49, 23, 33, 39, 17, 55, 1]
Subgroup Elements =  set([1, 17, 33, 7, 49, 23, 55, 39])
Subgroup Identity =  1
Subgroup Order =  8
Inverse of element 25 =  1
Multiplication of 7 and 25 =  15
18 to the power 3 =  15
Element:  1  has order:  1
Element:  17  has order:  2
Element:  33  has order:  2
Element:  7  has order:  4
Element:  49  has order:  2
Element:  23  has order:  4
Element:  55  has order:  4
Element:  39  has order:  4
\end{verbatim}
\subsection{Example~\ref{Ex8}}
\begin{verbatim}
Create a group (5,+)x(9,x)
Group elements:
Element  1  is  [0, 1]
Element  2  is  [0, 2]
Element  3  is  [0, 4]
Element  4  is  [0, 5]
Element  5  is  [0, 7]
Element  6  is  [0, 8]
Element  7  is  [1, 1]
Element  8  is  [1, 2]
Element  9  is  [1, 4]
Element  10  is  [1, 5]
Element  11  is  [1, 7]
Element  12  is  [1, 8]
Element  13  is  [2, 1]
Element  14  is  [2, 2]
Element  15  is  [2, 4]
Element  16  is  [2, 5]
Element  17  is  [2, 7]
Element  18  is  [2, 8]
Element  19  is  [3, 1]
Element  20  is  [3, 2]
Element  21  is  [3, 4]
Element  22  is  [3, 5]
Element  23  is  [3, 7]
Element  24  is  [3, 8]
Element  25  is  [4, 1]
Element  26  is  [4, 2]
Element  27  is  [4, 4]
Element  28  is  [4, 5]
Element  29  is  [4, 7]
Element  30  is  [4, 8]
Group Identity =  [0, 1]
Group Order =  30
Element 3 =  [0, 4]
Inverse of element 3 =  [0, 7]
Element 7 =  [1, 1]
Element 6 =  [0, 8]
Multiplication of element 7 and element 6 =  [1, 8]
Element 7 to the power of 3 =  [3, 1]
Group orders:
Element  [0, 1]  has order  1
Element  [0, 2]  has order  6
Element  [0, 4]  has order  3
Element  [0, 5]  has order  6
Element  [0, 7]  has order  3
Element  [0, 8]  has order  2
Element  [1, 1]  has order  5
Element  [1, 2]  has order  30
Element  [1, 4]  has order  15
Element  [1, 5]  has order  30
Element  [1, 7]  has order  15
Element  [1, 8]  has order  10
Element  [2, 1]  has order  5
Element  [2, 2]  has order  30
Element  [2, 4]  has order  15
Element  [2, 5]  has order  30
Element  [2, 7]  has order  15
Element  [2, 8]  has order  10
Element  [3, 1]  has order  5
Element  [3, 2]  has order  30
Element  [3, 4]  has order  15
Element  [3, 5]  has order  30
Element  [3, 7]  has order  15
Element  [3, 8]  has order  10
Element  [4, 1]  has order  5
Element  [4, 2]  has order  30
Element  [4, 4]  has order  15
Element  [4, 5]  has order  30
Element  [4, 7]  has order  15
Element  [4, 8]  has order  10
\end{verbatim}
\subsection{Example~\ref{Ex8}}
\begin{verbatim}
Create (6,+)x(9,x)
Group elements:
Element  1  is  [0, 1]
Element  2  is  [0, 2]
Element  3  is  [0, 4]
Element  4  is  [0, 5]
Element  5  is  [0, 7]
Element  6  is  [0, 8]
Element  7  is  [1, 1]
Element  8  is  [1, 2]
Element  9  is  [1, 4]
Element  10  is  [1, 5]
Element  11  is  [1, 7]
Element  12  is  [1, 8]
Element  13  is  [2, 1]
Element  14  is  [2, 2]
Element  15  is  [2, 4]
Element  16  is  [2, 5]
Element  17  is  [2, 7]
Element  18  is  [2, 8]
Element  19  is  [3, 1]
Element  20  is  [3, 2]
Element  21  is  [3, 4]
Element  22  is  [3, 5]
Element  23  is  [3, 7]
Element  24  is  [3, 8]
Element  25  is  [4, 1]
Element  26  is  [4, 2]
Element  27  is  [4, 4]
Element  28  is  [4, 5]
Element  29  is  [4, 7]
Element  30  is  [4, 8]
Element  31  is  [5, 1]
Element  32  is  [5, 2]
Element  33  is  [5, 4]
Element  34  is  [5, 5]
Element  35  is  [5, 7]
Element  36  is  [5, 8]
Create subgroup with elements 2 and 22
2 creates =  [[0, 2], [0, 4], [0, 8], [0, 7], [0, 5], [0, 1]]
22 creates =  [[3, 5], [0, 7], [3, 8], [0, 4], [3, 2], [0, 1]]
Subgroup elements:
Element  1  is  [0, 1]
Element  2  is  [3, 2]
Element  3  is  [3, 1]
Element  4  is  [0, 7]
Element  5  is  [3, 8]
Element  6  is  [0, 5]
Element  7  is  [0, 4]
Element  8  is  [3, 7]
Element  9  is  [0, 8]
Element  10  is  [3, 4]
Element  11  is  [0, 2]
Element  12  is  [3, 5]
Subgroup Identity =  [0, 1]
Subgroup Order =  12
Element 3 =  [3, 1]
Inverse of element 3 =  [3, 1]
Element 7 =  [0, 4]
Element 6 =  [0, 5]
Multiplication of element 7 and element 6 =  [0, 2]
Element 7 to the power of 3 =  [0, 1]
Group orders:
Element  [0, 1]  has order  1
Element  [3, 2]  has order  6
Element  [3, 1]  has order  2
Element  [0, 7]  has order  3
Element  [3, 8]  has order  2
Element  [0, 5]  has order  6
Element  [0, 4]  has order  3
Element  [3, 7]  has order  6
Element  [0, 8]  has order  2
Element  [3, 4]  has order  6
Element  [0, 2]  has order  6
Element  [3, 5]  has order  6
\end{verbatim}

\subsection{Example~\ref{Ex4}}
\begin{verbatim}
Testing group:
(32,x) group order =  16
(32,x) sorted element orders 
=  [1, 2, 2, 2, 4, 4, 4, 4, 8, 8, 8, 8, 8, 8, 8, 8] 

Candidate groups:
(16,+) sorted element orders 
= [1, 2, 4, 4, 8, 8, 8, 8, 16, 16, 16, 16, 16, 16, 16, 16]
(8,+)x(2,+) sorted element orders 
=  [1, 2, 2, 2, 4, 4, 4, 4, 8, 8, 8, 8, 8, 8, 8, 8]
(4,+)x(2,+)x(2,+) sorted element orders 
=  [1, 2, 2, 2, 2, 2, 2, 2, 4, 4, 4, 4, 4, 4, 4, 4]
(2,+)x(2,+)x(2,+)x(2,+) sorted element orders 
=  [1, 2, 2, 2, 2, 2, 2, 2, 2, 2, 2, 2, 2, 2, 2, 2] 

Conclusion:
(32,x) is therefore isomorphic to (8,+)x(2,+)
\end{verbatim}
\end{document}